\documentclass[a4paper,12pt]{article}
\usepackage[utf8]{inputenc}

\usepackage{amsmath}
\usepackage{amsthm}
\usepackage{amssymb}
\usepackage{amsfonts}
\usepackage{enumerate}
\usepackage{graphicx}
\usepackage{hyperref}

\theoremstyle{definition}
\newtheorem{teo}{Theorem}[subsection]
\newtheorem{lem}[teo]{Lemma}
\newtheorem{cor}[teo]{Corollary}
\newtheorem{oq}[teo]{Open question}
\theoremstyle{plain}
\newtheorem{defin}[teo]{Definition}
\theoremstyle{remark}

\newcommand{\sx}{\left}
\newcommand{\dx}{\right}
\newcommand{\R}{\mathbb{R}}

\newcommand{\hyp}{\textbf{H}}
\newcommand{\G}{\textbf{G}}
\newcommand{\F}{\mathbb{F}}

\newcommand{\N}{\mathbb{N}}
\renewcommand{\geq}{\geqslant}
\renewcommand{\leq}{\leqslant}
\renewcommand{\phi}{\varphi}
\newcommand{\eps}{\varepsilon}

\newcommand{\tonde}[1]{\left( #1 \right)}
\newcommand{\quadre}[1]{\left[ #1 \right]}
\newcommand{\graffe}[1]{\left\{ #1 \right\}}
\newcommand{\va}[1]{\left| #1 \right|}

\newcommand{\restr}[2]{\left. #1 \right|_{#2}}

\newcommand{\barr}[1]{\overline{#1}}
\newcommand{\wtde}[1]{\widetilde{#1}}
\newcommand{\CG}{\mathcal{C}}
\newcommand{\UF}{\mathcal{U}}

\DeclareMathOperator{\Imm}{Im}

\title{Quasi-Isometric Rigidity of Cusp-Decomposable Manifolds}
\author{Kirill Kuzmin}

\begin{document}

\maketitle

\begin{abstract}
In this paper we explore coarse properties of cusp-decomposable manifolds defined in \cite{NPcusp}. We describe the large scale geometry of the universal cover of a cusp-decomposable manifold and of quasi-isometries between two such universal covers. This description will provide us the tools to prove quasi-isometric rigidity for fundamental groups of cusp-decomposable manifolds.
\end{abstract}

\section*{Introduction}

This paper is concerned with cusp-decomposable manifolds and their quasi-isometric rigidity.

Let $n$ be an integer greater than or equal to $3$ and let $X$ be a connected, complete, non-compact and finite volume, locally symmetric, negatively curved $n$-manifold without boundary. We can perform the thin-thick decomposition on $X$; the thin part splits up in finitely many connected components called \emph{cusps}. For each cusp there is an $\tonde{n-1}$-manifold $N$ such that the cusp is diffeomorphic to $N \times \left[ 0, +\infty \right)$.

Now remove, for some constant $a>0$, each $N \times \left( a, +\infty \right)$ and call the resulting compact manifold with boundary $P$ a \emph{piece}; it is a compactification of $X$. Take any finite number of pieces, choose a pairing on a subset of boundary components such that there is an affine diffeomorphism between the paired boundary components and glue the paired boundary components along these diffeomorphisms; we will see later how to characterize affine diffeomorphisms in this setting. Two glued boundary components do not need to belong to different pieces. Suppose that the resulting manifold $M$ is connected. Then $M$ is a \emph{cusp-decomposable manifold}, according to the definition given in \cite{NPcusp}. We will always suppose there is at least one gluing; indeed, quasi-isometric rigidity results for single pieces are well known from \cite{SchwRank1}. It will turn out later that all the pieces are locally isometric to the same symmetric space.

The paper \cite{NPcusp} was concerned with smooth rigidity results for cusp-decomposable manifolds. Here we are going to explore these manifolds from the quasi-isometric point of view. To any cusp-decomposable manifold $M$ there is associated a tree constructed as follows. Let $\wtde{M}$ be the universal cover of $M$. The vertices correspond to connected components of the preimages of pieces of $M$ in $\wtde{M}$, called \emph{chambers}, and there is an edge between two vertices if the corresponding chambers intersect; we will use slightly different and more precise definitions later. The main results we will prove are the following:

\begin{teo}[Theorem \ref{chambersQuasiPreserved}, Corollary \ref{treeIso}]
Let $M_1$ and $M_2$ be two cusp-decomposable manifolds, and suppose there is a quasi-isometry $f$ between their universal covers $\wtde{M_1}$ and $\wtde{M_2}$. Then the image via $f$ of a chamber $C$ in $\wtde{M_1}$ is at a finite Hausdorff distance from a unique chamber in $\wtde{M_2}$ which is quasi-isometric to $C$. Moreover, the induced map between the vertices of the trees associated to $M_1$ and $M_2$ is the restriction of a tree isomorphism.
\end{teo}

\begin{teo}[Theorem \ref{HausGroupPieces}]
Let $M_1$ and $M_2$ be cusp-decomposable manifolds, and suppose there is a quasi-isometry $f$ between their fundamental groups $\pi_1 \tonde{M_1}$ and $\pi_1\tonde{M_2}$. Then the image via $f$ of the fundamental group of a piece $P_1$ in $M_1$ is at a finite Hausdorff distance from a conjugate of the fundamental group of a piece $P_2$ in $M_2$. Moreover, $\pi_1 \tonde{P_1}$ is quasi-isometric to $\pi_1 \tonde{P_2}$.
\end{teo}

\begin{cor}
Commensurability classes of pieces in $M_1$ are the same as commensurability classes for pieces in $M_2$.
\end{cor}

\begin{teo}[Theorem \ref{main}]
Let $\Gamma$ be a group quasi-isometric to the fundamental group of a cusp-decomposable manifold $M$. Then $\Gamma$ is virtually isomorphic to the fundamental group of a ``cusp-decomposable orbifold'', an object constructed similarly to a cusp-decomposable manifold with the difference that pieces are obtained from finite volume locally symmetric orbifolds.
\end{teo}

\begin{teo}[Theorem \ref{corMain}]
The class of orbifold fundamental groups of cusp-decomposable orbifolds is quasi-isometrically rigid, i.e. every group quasi-isometric to an orbifold fundamental group of a cusp-decomposable orbifold is actually virtually isomorphic to the fundamental group of a, maybe different, cusp-decomposable orbifold.
\end{teo}

Similar results were already proven for different classes of manifolds. In \cite{KaLeeb3manQI}, Kapovich and Leeb proved that quasi-isometries detect the JSJ decomposition of Haken $3$-manifolds, and showed that the class of fundamental groups of Haken $3$-manifolds is quasi-isometrically rigid. In \cite{FLSgraph}, Frigerio, Lafont and Sisto considered generalized graph manifolds. As in our case, these manifolds are built up from pieces glued along affine diffeomorphisms of their boundary components. Pieces are products of finite volume non-compact hyperbolic manifolds by tori. To prove results on quasi-isometries of universal covers of two generalized graph-manifolds an additional geometric hypothesis is needed, called ``irreducibility'' by the authors. In both papers a result of the type ``chambers go near to chambers'' was proven, that allowed to characterize groups quasi-isometric to the universal covers of the manifolds considered. Many of the ideas of this paper can be tracked back to these two works. We refer the reader to \cite{FrigSurvey} for a survey on quasi-isometric rigidity for fundamental groups of manifolds that decompose in geometric pieces.

The already cited work by Schwartz, which is heavily used in the proof of the main theorem, proves a stronger result for pieces: a group quasi-isometric to the orbifold fundamental group of a complete non-compact finite volume negatively curved locally symmetric orbifold is actually virtually isomorphic to it. So a natural question is the following:

\begin{oq}
Is a group quasi-isometric to the fundamental group of a cusp-decomposable manifold, or the orbifold fundamental group of a cusp-decomposable orbifold, actually virtually isomorphic to it?
\end{oq}

Other questions are related to how much the class of orbifold fundamental groups of cusp-decomposable orbifolds is larger than the class of fundamental groups of cusp-decomposable manifolds.

\begin{oq}
Consider the orbifold fundamental group of a cusp-decomposable orbifold. Is it necessarily virtually isomorphic, or at least quasi-isometric, to the fundamental group of a cusp-decomposable manifold?
\end{oq}

An affirmative answer would follow from the affirmative answer to the following question, more geometric in nature.

\begin{oq}
Does a cusp-decomposable orbifold necessarily have a finite cover which is a cusp-decomposable manifold?
\end{oq}

\section*{Aknowledgements}

The content of this paper is part of the upcoming PhD thesis I am writing under the supervision of Roberto Frigerio. I would like to thank him for the many inspiring conversations on the topics discussed here.

\section{Preliminaries}

We recall here some useful tools and set the notation.

\subsection{Notation}

The letter $e$ will usually denote the identity of a group $G$. If $g$ is an element of a group, we will denote by $g$ also the left multiplication by it, which may have additional structure, e.g. of a diffeomorphism for a Lie group. If $G$ acts on a set, the action of $g \in G$ will be denoted by $g \cdot$.

The letter $d$ will always denote the distance function in a metric space. By this letter we will also denote the distance $d \tonde{a,A}$ of a point $a$ from a subset $A$ in a metric space, and also the infimum $d \tonde{A_1,A_2}$ of distances between points $p_1 \in A_1$ and $p_2 \in A_2$, not to be confused with the Gromov-Hausdorff distance between $A_1$ and $A_2$. The $R$-neighbourhood of a subset $A$, for some non-negative real number $R$, is the set of all points at a distance at most $R$ from $A$.

The length of a piecewise smooth curve with respect to a given Riemannian metric will be denoted by $\ell$. Sequences indicized by $\N$ will be often denoted by $\tonde{a_i}$ rather than by $\tonde{a_i}_{i \in \N}$.

\subsection{Large scale geometry}

Let $E$, $F$ be metric spaces. A function $f \colon E \to F$ is a \emph{$\tonde{K,c}$-quasi-isometric embedding} for some costants $K \geq 1$ and $c \geq 0$ if the following inequality
\[ K^{-1}d\tonde{p,q} - c \leq d \tonde{f \tonde{p}, f \tonde{q}} \leq Kd\tonde{p,q} + c \]
holds for every $p$, $q$ in $E$.

A quasi-isometric embedding is said to be a \emph{quasi-isometry} if one of the following two equivalent conditions holds:

\begin{itemize}
\item \emph{Existence of a quasi-inverse}: there exists a quasi-isometric embedding $g \colon F \to E$ such that the distances $d \tonde{x, g \circ f \tonde{x}}$ and $d \tonde{y, f \circ g \tonde{y}}$ are uniformly bounded;
\item \emph{Quasi-surjectivity}: There exists a constant $D$ such that any point in $F$ is at a distance at most $D$ from some point in $f \tonde{E}$. In this case we say that $f \tonde{E}$ is $D$-dense in $F$.
\end{itemize}

Let $\Gamma$ be a group generated by a finite set $S$; we will assume for simplicity that $x$ belongs to $S$ if and only if $x^{-1}$ does. There is a standard way to put a meaningful distance function on $\Gamma$.

\begin{defin}
The \emph{Cayley graph} $\CG \tonde{\Gamma,S}$ of $\Gamma$ with respect to the generating set $S$ is a graph vith vertex set $\Gamma$ and an arc between $g$ and $h$ if and only if $g^{-1}h$ is in $S$.

A distance function on the Cayley graph is obtained by setting the length of any edge equal to $1$ and by taking the induced path distance; the restriction of this distance to $\Gamma$ is called \emph{word metric}, in fact $d\tonde{g,h}$ is the minimum length of a word of elements of $S$ representing $g^{-1}h$.
\end{defin}

The precise distance obviously depends on the generating set, but all Cayley graphs for a fixed group, and hence all word metrics, are quasi-isometric.

Quasi-isometries do not distinguish a finitely generated group from a finitely generated subgroup of finite index. Also they do not distinguish a finitely generated group from any of its quotients by finite normal subgroups. We sum up these facts in the following

\begin{defin}
Two finitely generated groups $G$ and $H$ are \emph{virtually isomorphic} if there is a finite sequence $G=G_0$, $G_1$, \dots, $G_m=H$ of finitely generated groups such that for any $i=1$, \dots, $m$ the groups $G_{i-1}$ and $G_i$ are obtained one from the other either by taking a finite index subgroup or by extension by a finite group.
\end{defin}

This means that virtually isomorphic groups cannot be distinguished by quasi-isometries.

Groups can act on metric spaces by quasi-isometries.

\begin{defin}
Fix non-negative real constants $K \geq 1$, $c$ and $D$. A \emph{$\tonde{K,c,D}$-quasi-action by quasi-isometries} of a group $\Gamma$ on a metric space $E$ is a map $\rho$ from $\Gamma$ to the set of quasi-isometries of $E$ such that:
\begin{itemize}
\item Every quasi-isometry in the image of $\Gamma$ is a $\tonde{K,c}$-quasi-isometry with $D$-dense image;
\item If $e$ is the identity element of $H$, the distances $d \tonde{x, \rho \tonde{e} \tonde{x}}$ are uniformly bounded;
\item For any $g$, $h$ in $\Gamma$, the distances $d \tonde{\rho\tonde{gh}\tonde{x},\rho \tonde{g}\tonde{\rho\tonde{h}\tonde{x}}}$ are uniformly bounded.
\end{itemize}
The quasi-action is said to be cobounded if, equivalently, the orbit of some point is $D'$-dense for some $D'$ or the orbit of \emph{any} point is $D''$-dense for some $D''$.
\end{defin}

An easy example of a cobounded quasi-action is the following (see \cite[Lemma 5.60]{DKlecGGT}):

\begin{lem}\label{QItoQIaction}
Suppose $\Gamma$ is a finitely generated group, $E$ is a metric space, $f \colon \Gamma \to E$ is a quasi-isometry and $\phi \colon E \to \Gamma$ is a quasi-inverse of $f$. Then the map that sends $g$ to the quasi-isometry $f \tonde{g\phi\tonde{\cdot}}$ is a cobounded quasi-isometric quasi-action.
\end{lem}

A useful example of quasi-isometry is provided by the following result which we will not state in full generality.

\begin{teo}[Milnor-\v{S}varc Lemma]
Let $M$ be a smooth compact manifold, possibly with boundary. Fix a Riemannian metric on it and then lift it to the universal cover $\wtde{M}$, and let $d$ be the associated distance function. Fix also a point $x_0$ in $\wtde{M}$. Then $\pi_1 \tonde{M}$ is finitely generated and the map from $G$ to $\wtde{M}$ sending $g$ to $g \cdot x_0$, where the action is by deck transformations, is a quasi-isometry between $G$ with any word metric and $\tonde{\wtde{M},d}$.
\end{teo}

This means that up to quasi-isometry, we can consider universal covers of compact manifolds as equivalent to their fundamental groups. This in particular is true for locally symmetric pieces as defined above, and for cusp-decomposable manifolds too.

\subsection{Ultrafilters and asymptotic cones}

We now provide some tools from mathematical logic, which will be used later to define asymptotic cones.

\begin{defin}
A family $\mathcal{F}$ of subsets of an index set $I$ is called a \emph{filter} if
\begin{itemize}
\item The set $I$ belongs to $\mathcal{F}$ and the empty set does not;
\item $\mathcal{F}$ is closed under supersets and finite intersections.
\end{itemize}

A filter $\UF$ is called an \emph{ultrafilter} if for every pair $A$, $B$ of subsets of $I$ such that $A \cup B$ belongs to $\UF$, at least one of them belongs to $\UF$ too.

An ultrafilter is said to be \emph{non-principal} if it contains all the subsets of $I$ with finite complement. 
\end{defin}

\begin{defin}
Let $\UF$ be an ultrafilter on a set $I$, let $A$ be a set and $\tonde{a_i}_{i \in I}$ a sequence of elements of $A$. We say that a property of elements of $A$ holds $\UF$-\emph{almost everywhere} on $\tonde{a_i}$, shortly $\UF$-a.e., if the subset of indices $j$ such that the property is true for $a_j$ is in $\UF$.
\end{defin}

We give the notion of limit with respect to an ultrafilter.

\begin{defin}
Let $\UF$ be an ultrafilter on $\N$. We say that a sequence $\tonde{a_i}_{i \in \N}$ of real numbers is \emph{$\UF$-a.e. bounded} if there exists a non-negative real number $B$ such that $\va{a_i} \leq B$ $\UF$-a.e. We say that the \emph{$\UF$-limit} of such a sequence is $L$, shortly $\UF-\lim a_i = L$, if for any positive $\eps$ the inequality $\va{a_i - L}<\eps$ holds $\UF$-a.e.
\end{defin}

The following lemma shows that $\UF$-limits behave better than classical limits (see \cite[Lemma 7.23]{DKlecGGT}).

\begin{lem}
Every $\UF$-bounded sequence of real numbers has a unique $\UF$-limit. If $\UF$ is non-principal and the sequence converges in the classical sense, the $\UF$-limit is equal to the classical limit. $\UF$-limits behave well with respect to algebraic operations, composition with continuous functions, etc., as classical limits do.
\end{lem}

We can now define the asymptotic cone of a metric space.

\begin{defin}
Let $E$ be a metric space. If $\lambda$ is a real number, we denote by $\lambda E$ the metric space with the same support as $E$ and with all distances multiplied by $\lambda$. Let furthermore $\UF$ be a non-principal ultrafilter on $\N$, let $\tonde{\lambda_i}_{i \in \N}$ be a sequence of positive real numbers which is not $\UF$-a.e. bounded and $\tonde{x_i}_{i \in \N}$ be a sequence of points in $E$. The asymptotic cone $\CG \tonde{E, \tonde{x_i}, \tonde{\lambda_i}, \UF}$, or simply $\CG E$ if the construction data are clear or inessential, is the quotient of the set of sequences $\tonde{y_i}$ of points in $E$ such that $\lambda_i^{-1} d \tonde{x_i,y_i}$ is $\UF$-a.e. bounded by the relation that identifies two sequences $\tonde{y_i}$ and $\tonde{z_i}$ if $\UF-\lim \lambda_i^{-1} d \tonde{y_i,z_i} = 0$. The set $\CG E$ is endowed with a metric space structure, the distance function being given by $d \tonde{\quadre{y_i},\quadre{z_i}} = \UF-\lim \lambda_i^{-1} d \tonde{\tonde{y_i},\tonde{z_i}}$.
\end{defin}

In order to study the structure of a metric space via asymptotic cones it might be useful to see how the decomposition of the space in certain subsets induces a decomposition of its asymptotic cones in related subsets. To this aim we introduce the following

\begin{defin} \label{asFamily}
Let $E$ be a metric space. Consider a non-principal ultrafilter $\UF$ on $\N$, a sequence $\tonde{x_i}$ of points in $E$ and a sequence $\tonde{\lambda_i}$ of positive real numbers which is not $\UF$-a.e. bounded. Use this data to construct an asymptotic cone $\CG E$ of $E$.

Let $\mathcal{A}$ be a family of subsets of $E$, and let $\tonde{A_i}_{i \in \N}$ be a sequence of sets in $\mathcal{A}$ such that $\lambda_i^{-1} \tonde{d \tonde{x_i,A_i}}$ is $\UF$-a.e bounded. Consider the set of all classes of sequences $\quadre{y_i}$ in $\CG E$ such that $y_i$ belongs to $A_i$. They form a subset of $\CG E$. Consider the family of all subsets of $\CG E$ obtained this way by varying the sequence $\tonde{A_i}$. We call it the \emph{asymptotic family} of $\mathcal{A}$ and denote it by $\CG \mathcal{A}$.
\end{defin}

Note that if every point of $E$ is at a bounded distance from the union of $\mathcal{A}$ then $\CG E$ is the union of the sets in $\CG \mathcal{A}$.

Asymptotic cones turn useful in the study of quasi-isometries thanks to the following result (see \cite[Lemma 7.69]{DKlecGGT}).

\begin{lem} \label{QItoBL}
Let $E_1$ and $E_2$ be metric spaces. Fix a non-principal ultrafilter $\UF$ on $\N$, a sequence of positive real numbers $\tonde{\lambda_i}$ which is not $\UF$-a.e. bounded and a sequence $\tonde{x_i}$ of points in $E_1$. Suppose that for every $i \in \N$ there is a $\tonde{K,c}$-quasi-isometric embedding $f_i \colon E_1 \to E_2$. Then the function
\[ \CG f \colon \CG \tonde{E_1, \tonde{x_i}, \tonde{\lambda_i}, \UF} \to \CG \tonde{E_2, \tonde{f \tonde{x_i}}, \tonde{\lambda_i}, \UF} \]
that sends the class of a sequence $\tonde{y_i}$ to the class of the sequence $\tonde{f_i \tonde{y_i}}$ is well defined and is a $K$-bi-Lipschitz embedding. 
\end{lem}

\subsection{Geodesic metric spaces}

Recall that a geodesic is an isometric embedding of an interval of the real line into a metric space.

\begin{defin}
A metric space is called \emph{geodesic} if for any pair $p$ and $q$ of its points there is a geodesic $\gamma$ that joins $p$ and $q$.
\end{defin}

By the Hopf-Rinow Theorem, all complete Riemannian manifolds are geodesic. We will often confuse for brevity a geodesic with its image.

The property of being geodesic passes to asymptotic cones (see \cite[Proposition 7.54]{DKlecGGT}).

\begin{lem}
Any asymptotic cone of a cusp-decomposable manifold is geodesic.
\end{lem}

Geodesics allow to define convexity.

\begin{defin}
A subset of a geodesic metric space is said to be \emph{convex} if any geodesic between any two of its points is contained in the subset.
\end{defin}

We now introduce the class of tree-graded geodesic spaces, that often arises when studying asymptotic cones. We first define geodesic triangles.

\begin{defin}
A \emph{geodesic triangle} of vertices $p$, $q$, $r$ in a geodesic metric space is the union of a geodesic between $p$ and $q$, a geodesic between $q$ and $r$ and a geodesic between $r$ and $p$. These three geodesics are called \emph{sides} of the triangle.

A geodesic triangle is said to be \emph{simple} if two distinct sides intersect only at endpoints.
\end{defin}

We can now recall the following definition from \cite[Definition 1.10]{DrutuSapir}.

\begin{defin} \label{topTreeGraded}
A geodesic metric space is \emph{tree-graded} with respect to a family of closed convex subspaces if:
\begin{enumerate}[(1)]
\item The space is the union of the subspaces of the family;
\item The intersection of two subspaces in the family contains at most one point;
\item Every simple geodesic triangle is contained in one subspace of the family.
\end{enumerate}
\end{defin}

We will often use the following lemma (see \cite[Lemma 2.13]{DrutuSapir}).

\begin{lem}
Let $E$ be a geodesic metric space tree-graded with respect to a family $\mathcal{A}$ of subspaces. Suppose that subspaces in $\mathcal{A}$ do not have global cut-points. Then the image of any topological embedding of a path-connected space without cut-points is contained in one subspace of the family.
\end{lem}


We now describe an important class of geodesic metric spaces (see \cite[Definition 3.19]{DrutuSapir}.

\begin{defin}
A geodesic metric space $E$ is asymptotically tree-graded with respect to a family $\mathcal{A}$ of subsets if any its asymptotic cone $\CG E$ is tree-graded with respect to the asymptotic family $\CG \mathcal{A}$ defined in \ref{asFamily}.
\end{defin}

When applied to groups, the previous definition becomes one of the many equivalent definitions of a relatively hyperbolic group \cite[Theorem 8.5]{DrutuSapir}.

\begin{defin}
A finitely generated group $G$ is hyperbolic relatively to a family of finitely generated subgroups $\graffe{H_1,\dots,H_m}$ if the Cayley graph of $G$ is asymptotically tree-graded with respect to the family of left cosets of the $H_i$'s.
\end{defin}

\subsection{Locally symmetric spaces}

Let $X$ be a connected, complete, non-compact and finite volume, locally symmetric, negatively curved $n$-manifold, with $n$ greater than or equal to $3$. Its universal cover $\wtde{X}$ is isometric to one of the following symmetric spaces (see \cite[\S 2.3]{SchwRank1}):

\begin{itemize}
 \item The real hyperbolic space $\R \hyp^n$;
 \item The complex hyperbolic space $\mathbb C \hyp^{\frac{n}{2}}$; this case can occur only if $n$ is even;
 \item The quaternionic hyperbolic space $\mathbb H \hyp^{\frac{n}{4}}$; this case can occur only if $n$ is a multiple of $4$;
 \item The Cayley or octonionic hyperbolic plane $\mathbb O \hyp^2$; this case can occur only if $n$ is equal to $16$.
\end{itemize}

We will concisely refer to the geometry of the universal cover by saying that $X$ is $n$-dimensonal $\F$-hyperbolic, with $\F$ equal to, respectively, $\R$, $\mathbb{C}$, $\mathbb{H}$ or $\mathbb{O}$.

We recall now some definitions in the non-positively curved setting. See e.g. \cite{BGS}, \cite{BrHa} for more details.

\begin{defin}
Let $Y$ be a simply connected complete Riemannian $m$-manifold of non-positive curvature. The \emph{boundary at infinity} of $Y$, denoted by $\partial Y$, is the quotient of the set of geodesic rays $\gamma \colon \sx[ 0, +\infty \dx) \to Y$ by the relation that identifies two of them if their images lie at finite Hausdorff distance from each other. The points of $\partial Y$ are called \emph{points at infinity} of $Y$.

The union of $Y$ and $\partial Y$ is denoted by $\barr{Y}$ and has a topology which coincides with the usual one on $Y$ and makes $\barr{Y}$ homeomorphic to a closed $m$-dimensional ball.
\end{defin}

We will need the following result (see \cite[\S 3.2]{BGS}, \cite[\S 4.13]{BGS}).

\begin{lem}
For any pair of points $p$ in $Y$ and $\omega$ in $\partial Y$ there is a unique geodesic ray $\gamma \colon \sx[ 0, +\infty \dx) \to Y$ with $\gamma \tonde{0} = p$ which is in the class of $\omega$.

If furthermore $Y$ has sectional curvatures bounded above by some negative real number, then for any pair of points $\omega_1$ and $\omega_2$ in $\partial Y$ there is a unique, up to reparametrization via translations, geodesic line $\gamma \colon \R \to Y$ such that the ray defined on the positive half-line of $\R$ is in the class of $\omega_1$ and the ray defined on the negative half-line is in the class of $\omega_2$.
\end{lem}

\begin{defin}
Let $Y$ be a simply connected complete Riemannian manifold of non-positive curvature. Choose a geodesic ray $\gamma$ in $Y$. The \emph{Busemann function} relative to $\gamma$ is the real valued function $b \colon Y \to \R$ defined by $b \tonde{p} = \lim_{t \rightarrow +\infty} d\tonde{p,\gamma\tonde{t}}-t$.
\end{defin}

Busemann functions relative to equivalent rays differ by a constant. It is thus possible to talk about Busemann functions relative to a point in the boundary at infinity. When the sectional curvature is bounded above by a negative real number, the limit of a Busemann function relative to $\omega$ at a point at infinity in $\barr{Y}$ is $-\infty$ at $\omega$ and $+\infty$ at all the other points. Furthermore, Busemann functions are convex and their gradient has norm equal to $1$.

\begin{defin}
A \emph{horoball} is a sublevel set of a Busemann function. A horosphere is the boundary of a horoball.
\end{defin}

When the sectional curvature is bounded above by a negative real number, every horoball and every horosphere has a unique point at infinity in its closure in $\partial Y$, which is the point where the defining Busemann function tends to $-\infty$.

Now we return to our finite volume locally symmetric manifold $X$. If we required ``constant curvature'' instead of locally symmetric, we would have gotten a real hyperbolic manifold, so our treatment is a generalization of the purely hyperbolic case treated in \cite{FLSgraph}. As in real hyperbolic case, it is possible to perform the thin-thick decomposition on $X$, and this allows us to detect a finite number of cusps of the form $N \times \left[ 0, +\infty \right)$.

If we now remove from $X$ the internal part of all the cusps we get a compact manifold with boundary $P$, such that for any boundary component $N$ the induced map $\pi_1 \tonde{N} \to \pi_1 \tonde{P}$ is injective. We will call such a manifold a \emph{piece}, and its universal cover a \emph{neutered space}. The neutered space that covers $P$ will be denoted with $\wtde{P}$. It is obtained by removing from $\wtde{X}$ a collection of disjoint horoballs which is invariant under the action of $\pi_1 \tonde{P}$. With a slight abuse of notation, we will continue to say that $P$ is $n$-dimensional $\F$-hyperbolic when $X$ is.

Every boundary component $N$ is covered by a $\pi_1 \tonde{P}$-invariant collection of horospheres, and the metric tensor of $X$ induces on any such horosphere $\wtde{N}$ a Riemannian structure isometric to a nilpotent Lie group with a left invariant Riemannian metric. A boundary component $N$ is thus isometric to the quotient of a nilpotent Lie group by a discrete group of isometries generated by a finite number of Lie group automorphisms and left translations; such manifolds are called infranil.

The precise nilpotent group which covers $N$ depends on the geometry of $\wtde{X}$. For the real hyperbolic case, it is simply $\R^{n-1}$. If instead $\F$ is equal to $\mathbb C$, $\mathbb H$ or $\mathbb O$, let $k$ be the dimension of $\F$ as a real vector space, and let $\Imm \F$ be its imaginary part, which is a $\tonde{k-1}$-dimensional real vector space. Then the nilpotent Lie group that covers $N$ is isomorphic to the group $\G \tonde{\frac{n}{k}-1,\F}$, which is defined as follows (see \cite[\S 2.5]{SchwRank1}).
\begin{defin}
Let $m$ be a positive integer and $\F$ one of $\mathbb C$, $\mathbb{H}$ or $\mathbb{O}$. The simply connected nilpotent Lie group $\G \tonde{m,\F}$ is $\F^m \times \Imm \F$ with the following group operation:
\[ \tonde{z_1, \dots, z_m,s}*\tonde{w_1, \dots, w_m,t} = \]
\[ = \tonde{z_1 + w_1, \dots, z_m + w_m,s+t+\sum_{i=1}^m \Imm \barr{z_i}w_i}. \]
\end{defin}

It is easy to see that the inverse of $\tonde{z_1, \dots, z_m,s}$ is $\tonde{-z_1, \dots, -z_m,-s}$. The commutator subgroup and the center of $\G \tonde{m,\F}$ are both equal to $\Imm \F$, and thus the group is step $2$ nilpotent. For brevity, the Lie group which covers the boundary components of $P \subseteq X$ will be denoted by $\G \tonde{X}$.

Note that the group of isometries of $X$ acts freely and properly discontinuously on the covering symmetric space. We will need a more general definition in the case where the action is not necessarily free.

\begin{defin} \label{lsOrb}
Let $Y$ be a negatively curved symmetric space. Suppose that a group of isometries acts on $Y$ properly discontinuously, but not necessarily freely, and the quotient $X'$ by the action is not compact and has finite volume. Then $X'$ is a \emph{complete non-compact finite volume locally symmetric negatively curved orbifold}.
\end{defin}

All of the previous considerations apply to finite volume negatively curved locally symmetric orbifolds as well with slight modifications. In particular,

\begin{itemize}
\item It is possible to perform the thin-thick decomposition on $X'$ to get a finite number of cusps of the form $N' \times \left[ 0, +\infty \right)$;
\item By removing the internal part of all the cusps we get an orbifold with boundary $P'$. Every boundary component $N'$ is an orbifold, and the induced map $\pi_{1,orb} \tonde{N'} \to \pi_{1,orb} \tonde{P'}$ between \emph{orbifold} fundamental groups is injective;
\item The orbifold universal cover of an orbifold piece is still a neutered space, which is obtained by removing from $Y$ a collection of disjoint open horoballs which is invariant under the action of $\pi_{1,orb} \tonde{P'}$;
\item Boundary components $N'$ of $P'$ are covered by horospheres in $Y$, so their orbifold universal cover is the same as in the manifold case.
\end{itemize}

\subsection{Nilpotent groups}

Here we recall some facts on nilpotent groups. 
When talking about a Lie group, we will endow it with a left-invariant Riemannian metric; we will also introduce another useful metric on nilpotent Lie groups.

\begin{defin}
A closed manifold is said to be \emph{infranil} if it is the quotient of a simply connected nilpotent Lie group by any group of isometries acting freely and properly discontinuously.
\end{defin}

\begin{defin}
A group acting by isometries properly discontinuously and cocompactly, but not necessarily freely, on a simply connected nilpotent Lie group, is said to be \emph{almost-crystallographic}.
\end{defin}

A boundary component $N$ of a piece is an infranil manifold; its fundamental group $\pi_1 \tonde{N}$ is thus isomorphic to a discrete subgroup of the group of isometries of a suitable simply connected nilpotent Lie group $\G$. 
In the more general case of orbifold pieces any boundary component is a quotient of a simply connected nilpotent Lie group $\G$ by the action of an almost-crystallographic group, which is thus the orbifold fundamental group of that boundary component. It follows by Milnor-\v{S}varc Lemma that this almost-crystallographic group is quasi-isometric to $\G$.

\begin{defin}
A map between simply connected nilpotent Lie groups is said to be \emph{affine} if it is the composition of a Lie groups homomorphism with a left multiplication by an element of the target group. A map between orbifolds, in particular manifolds, whose universal covers are simply connected nilpotent Lie groups is affine if any its lift to universal covers is affine.
\end{defin}

We will need the following result (see \cite[Theorem 3.7]{Dekimpe}).

\begin{lem} \label{isoFromAff}
Let $N_1$ and $N_2$ be quotients of simply connected nilpotent Lie groups $\G_1$ and $\G_2$ respectively by an isometric action of almost-crystallographic groups $\Gamma_1$ and $\Gamma_2$ respectively. Suppose that there is an isomorphism $\phi \colon \Gamma_1 \to \Gamma_2$. Then $\phi$ is induced by an affine homeomorphism between $N_1$ and $N_2$. More precisely, there is an affine isomorphism $\barr{\phi} \colon \G_1 \to \G_2$ such that for any $x$ in $\G_2$ and for any $g$ in $\Gamma_1$ the equality $\phi \tonde{g} \cdot x = \barr{\phi} \tonde{g \cdot {\barr{\phi}}^{-1} \tonde{x}}$ holds.
\end{lem}

%
%
%
%


The Carnot-Caratheodory metric is a way to put a distance on a manifold in the subriemannian setting. Here we recall the definition specifically for the Lie groups $\G$ used in this paper.

\begin{defin}
Let $e \in \G = \G \tonde{m,\F}$ be the identity element. Fix a positive definite scalar product $b \tonde{\cdot,\cdot}$ on $V=T_e \F^m \subseteq T_e \G$. Consider the distribution of subspaces on $\G$ obtained via left translations of $\tonde{V,b}$. With an abuse of notation we will still call $V$ the distribution and $b$ the scalar product on it. Then for any pair of points $g_1$ and $g_2$ of $\G$ there is a piecewise smooth curve joining them everywhere tangent to the distribution. The function
\[ d_{CC} \tonde{g_1, g_2} = \inf \graffe{\int \sqrt{b \tonde{\gamma' \tonde{t}, \gamma' \tonde{t}}} dt }, \]
where the infimum is taken over all piecewise smooth curves $\gamma$ with endpoints $g_1$ and $g_2$ and everywhere tangent to $V$, is a \emph{Carnot-Caratheodory} distance function on $\G$, invariant by the action of $\G$ itself via left translations.
\end{defin}

Introducing this new distance has, however, no effect from the topological and the quasi-isometry point of view \cite[\S 2.6]{SchwRank1}.

\begin{lem} \label{dCCsameTop}
The $d_{CC}$ distance induces the usual topology $\G$. The identity of a nilpotent Lie group as above is a quasi-isometry between its Riemannian and Carnot-Caratheodory distances.
\end{lem}

It is known from \cite{GroNilp} that finitely generated nilpotent groups have locally compact asymptotic cones. We now want to study asymptotic cones of $\pi_1 \tonde{N}$, where $N$ is the boundary of a piece modelled on a locally symmetric space $X$, which by Milnor-\v{S}varc Lemma are the same as asymptotic cones of $\G \tonde{X}$. The work of Pansu \cite{PansuCroissance} gives an explicit description for any asymptotic cone of a simply connected nilpotent Lie group $G$; it is a suitable nilpotent Lie group $G_{\infty}$ constructed from $G$ and endowed with a Carnot-Caratheodory metric. In our case, this description reduces to the following result (see \cite[(B), \S 49]{PansuCroissance}).

\begin{lem} \label{asConesNilp}
Asymptotic cones of a nilpotent Lie group $\G \tonde{m,\F}$ are all isometric to the group itself with a left-invariant Carnot-Caratheodory metric.
\end{lem}

The following lemma will be useful in distinguishing cusp-decomposable manifolds.

\begin{lem} \label{DistNilpGroups}
Consider two nilpotent Lie groups $\G_1$ and $\G_2$, each being either some $\R^m$ or some $\G \tonde{m,\F}$ and endowed with a left-invariant Carnot-Caratheodory metric. Suppose that there is an open subset of $\G_1$ bi-Lipschitz homeomorphic to an open subset $\G_2$. Then $\G_1$ and $\G_2$ are isomorphic as Lie groups.
\end{lem}

\begin{proof}
Bi-Lipschitz homeomorphisms preserve topological and Hausdorff dimension. These dimensions are both $m$ for open subsets of $\R^m$. Take now $\F$ to be one of $\mathbb{C}$, $\mathbb{H}$ or $\mathbb{O}$ and let $k$ be the dimension of $\F$ as a real vector space. Then the topological dimension of an open subset of $\G \tonde{m,\F}$ is $km+k-1$ and the Hausdorff dimension is $km+2\tonde{k-1}$ \cite{MitchellCC}. The lemma easily follows.
\end{proof}


\section{Metric structure of a cusp-decomposable manifold}

We now study the structure of a cusp-decomposable manifold $M$ as a metric space.

\subsection{The gluings}

Remember that a cusp-decomposable manifold $M$ is the result of gluing a finite number of pieces along some pairs of their infranil boundary components identified via affine diffeomorphisms. Lemma  \ref{isoFromAff} easily leads to the following

\begin{lem}
In a cusp-decomposable manifold all the pieces are locally isometric to the same symmetric geometry.
\end{lem}

We can then denote by $\G \tonde{M}$ the nilpotent Lie group covering the boundary components of pieces of $M$.

We now want to put a Riemannian metric on a cusp-decomposable manifold. By the Milnor-\v{S}varc Lemma two such metrics, when lifted to the universal cover, will lead to distance functions quasi-isometric to the fundamental group, and thus quasi-isometric to each other. However, we want to choose a specific metric which is more confortable for some of the proofs below.

Remember that cusps of a piece $P$ are diffeomorphic to a product $N \times \sx[ 0, +\infty \dx)$, where $N$ is an infranil manifold. We can choose this identification in such a way that for each $a \geq 0$ the layer $N \times \graffe{a}$ is covered by horospheres in $\wtde{P}$. Now let $N_1 \times \sx[ 0, +\infty \dx)$ and $N_2 \times \sx[ 0, +\infty \dx)$ be two cusps which we wish to glue in $M$ along a diffeomorphism $\phi \colon N_1 \to N_2$; let $g_1$ and $g_2$ be the respective metric tensors on the two cusps. We remove $N_i \times \sx[ 3, +\infty \dx)$ from $N_i \times \sx[ 0, +\infty \dx)$ for $i=1$, $2$ and then identify $\tonde{n,a}$ with $\tonde{\phi \tonde{n},3-a}$ for each $n$  in $N_1$ and each $a$ strictly between $0$ and $3$. If a cusp $N' \times \sx[ 0, +\infty \dx)$ is not glued to another, we only remove $N' \times \sx( 3, +\infty \dx)$. The result of all these operations is obviously diffeomorphic to $M$.

Let $\psi \colon \tonde{0,3} \to \R$ be a smooth function constantly equal to $1$ on $\sx( 0, 1 \dx]$, strictly decreasing in $\quadre{1,2}$ and constantly equal to $0$ on $\sx[ 2,3 \dx)$. On the complement of identified portions of pieces we put the Riemannian metric of the piece. On a point $\tonde{p,a} \sim \tonde{\phi \tonde{p},3-a}$ of the identified portions we set (with some abuse of notation) the Riemannian metric to be equal to $\psi\tonde{a}\restr{g_1}{\tonde{p,a}} + \psi\tonde{3-a}\restr{g_2}{\tonde{\phi\tonde{p},3-a}}$.

\subsection{The geometry of the universal cover}

\label{ch-W-TW}

Let $M$ be a cusp-decomposable manifold and let $P$ be a piece of $M$. Denote by $\wtde{M}$ the universal cover of $M$. The piece $P$ is obtained from a finite volume non-compact locally symmetric manifold $X$, whose cusps are of the form $N \times \sx[ 0, +\infty \dx)$. The part $N \times \sx( 3, +\infty \dx)$ of a cusp is removed during the construction of $M$, and the collar $N \times \sx[ 0, 3 \dx]$ of the resulting boundary component might or might not be identified with another such collar. In the second case we get a boundary component $N \times \sx\{ 3 \dx\}$ of $M$.

We call a connected component of the preimage in $\wtde M$ of $X$ without the union of its cusps of the form $N \times \sx( 0, +\infty \dx)$ a \emph{chamber}, and a connected component of the preimage of a collar $N \times \sx[ 0, 3 \dx]$ a \emph{wall} adjacent to the chamber. A chamber intersects a wall adjacent to it in one or two connected components of the preimage of $N \times \sx\{ 0 \dx\}$. One such component will be called \emph{thin wall}. With the Riemannian metric restricted to the chamber, a thin wall is a horosphere in the chamber itself, and is thus metrically convex. A wall covering the collar of a boundary component in $M$ is a collar of a boundary component in $\wtde{M}$, and will be called \emph{boundary wall}.

Consider the metric tensor on $M$ defined at the end of the previous subsection. It lifts to the universal cover $\wtde M$ via pull-back. If not explicitly stated otherwise, from now on we will consider on $\wtde M$ the path distance induced by this metric tensor, and on submanifolds, such as chambers, walls and thin walls, the path distance induced by the restriction of this Riemannian metric. This distance is greater than or equal to the restriction of the distance of $\wtde{M}$.

\begin{lem} \label{finIsoClasses}
In the universal cover $\wtde{M}$ of a cusp-decomposable manifold $M$ there is only a finite number of isometry classes of chambers, walls and thin walls.
\end{lem}

\begin{proof}
The proof follows easily from the finiteness of pieces in $M$.
\end{proof}

We introduce now some structural constants of $\wtde{M}$ which will turn useful in the sequel.

\begin{lem} \label{DWDTW}
There is a constant $D_W$ depending only on $\wtde{M}$ such that the distance between every pair of distinct walls is at least $D_W$.
\end{lem}

\begin{proof}
The constant $D_W$ is the length of the shortest path in $M$ having endpoints on collars separating pieces and not homotopic into them relatively to the endpoints. The thesis follows by compactness of $M$.
\end{proof}

\begin{cor} \label{locFinWalls}
The family of walls in $\wtde{M}$ is locally finite, i.e. every closed ball $\mathcal{B}$ in $\wtde{M}$ intersects only a finite number of them.
\end{cor}

\begin{proof}
The ball $\mathcal{B}$ is compact. For every wall intersecting $\mathcal{B}$ take a point in the intersection and consider an open ball of radius $\frac{D_W}{2}$ centered in it. All these balls are disjoint by the previous lemma; by compactness, they are finite.
\end{proof}

The way the gluings are performed allows us to describe the fundamental group of a cusp-decomposable manifold $M$. Remember that fundamental groups of the boundary components of a piece embed into the fundamental group of the piece itself.

\begin{defin}
In the fundamental group of a piece of $M$ consider the family of the images of the fundamental groups of boundary components and their conjugates. We will call them \emph{cusp subgroups}. These subgroups also embed in the fundamental group of $M$. We will refer to their images and the conjugates of their images inside the whole $\pi_1 \tonde{M}$ with the same name.
\end{defin}

The following lemma is an immediate consequence of the definitions.

\begin{lem} \label{FundGroupGraph}
The fundamental group of $M$ is the fundamental group of a finite graph of groups where any vertex group is the fundamental group of a piece, and any edge group embeds as a cusp subgroup in the vertex groups associated to the endpoints of the edge.
\end{lem}

At the universal cover level, this graph structure leads us to the following

\begin{defin} \label{BassSerreTree}
The Bass-Serre tree of a cusp-decomposable manifold $M$ is a graph with a vertex for any chamber in $\wtde{M}$ and an edge between two vertices if and only if there is a wall touching both the corresponding chambers.
\end{defin}

We now explore the geometric structure of the universal cover of $M$.

\begin{lem} \label{biLipWalls} 
There are real constants $K_P$, $K_E$ greater than or equal to $1$ and depending only on $\wtde{M}$ such that:
\begin{itemize}
\item Every wall of $\wtde{M}$ is $K_P$-bi-Lipschitz homeomorphic to the product of $\G \tonde{M}$ with a real interval;
\item Every thin wall is $K_E$-bi-Lipschitz embedded in its wall.
\end{itemize}
\end{lem}

\begin{proof}
The first statement readily follows from the compactness of $M$.

The proof of the second part is analogous to that of the Lemma 7.2 in \cite{FLSgraph} and uses Milnor-\v{S}varc Lemma. The only difference is that here we have a $\pi_1 \tonde{N}$ action instead of a $\mathbb{Z}^{n-1}$ action, where $N$ is the infranil manifold defined as above.
\end{proof}

\subsection{Large scale structure and relative hyperbolicity}

In this section we study the large scale geometry of the universal cover $\wtde{M}$ of a cusp-decomposable manifold $M$. Our results will descend from the study of asymptotic cones of $\wtde{M}$. Thanks to Milnor-\v{S}varc Lemma and Lemma \ref{QItoBL}, these cones are bi-Lipschitz homeomorphic to suitable asymptotic cones of $\pi_1 \tonde{M}$.

In $\wtde{M}$ we have the family of its walls. Consider an asymptotic cone $\CG\wtde{M}$ of $\wtde{M}$. Via Definition \ref{asFamily}, the family of walls gives rise to the family of asymptotic walls in $\CG\wtde{M}$.

In the fundamental group $\pi_1\tonde{M}$ we choose a representative for each conjugacy class of cusp subgroups. They form a finite family $\mathcal{H}$. Consider now the family of all the left cosets of subgroups in $\mathcal{H}$. The quasi-isometry of $\pi_1 \tonde{M}$ with $\wtde{M}$ given by Milnor-\v{S}varc lemma takes any such coset at a finite Hausdorff distance from a wall of $\wtde{M}$. Viceversa, the fact that every conjugacy class of cusp subgroups has a representative in $\mathcal{H}$ tells us that every wall is at a finite Hausdorff distance from the image of a left coset of a group in $\mathcal{H}$. So, aside with asymptotic walls, we consider the family of asymptotic cosets of the subgroups in $\mathcal{H}$ in an asymptotic cone of $\pi_1 \tonde{M}$. 

\begin{lem}
The union of the asymptotic walls in an asymptotic cone of $\wtde{M}$ is the whole cone, and the union of the asymptotic cosets in an asymptotic cone of $\pi_1 \tonde{M}$ is again the whole cone.
\end{lem}

\begin{proof}
The first assertion follows from the remark after Definition \ref{asFamily} and the compactness of $M$. The second follows from the quasi-isometry given by Milnor-\v{S}varc Lemma, the fact that this quasi-isometry takes walls at a finite Hausdorff distance from cosets, and Lemma \ref{QItoBL}.
\end{proof}

We can then state the following important result, which will describe the structure of an asymptotic cone of $\pi_1 \tonde{M}$ with respect to the family of asymptotic cosets, and hence of an asymptotic cone of $\wtde{M}$ with respect to the family of asymptotic walls.

\begin{teo}
The fundamental group of a cusp-decomposable manifold $M$ is hyperbolic relatively to its cusp subgroups. As a consequence, any asymptotic cone of its universal cover $\wtde{M}$ is tree-graded with respect to the family of its asymptotic walls, which are bi-Lipschitz homeomorphic to a suitable nilpotent Lie group with a Carnot-Caratheodory metric.
\end{teo}

\begin{proof}
The first thesis is a consequence of the Dahmani combination Theorem \cite[Theorem 0.1 (1)]{DahmaniCombConvGrps}. We have to verify that Dahmani's Theorem applies in our case of interest, i.e. that the following conditions hold:
\begin{itemize}
\item The action of $\pi_1 \tonde{M}$ on the Bass-Serre tree of $\wtde{M}$ is 2-acylindrical. We briefly recall here the argument contained in the proof of \cite[Lemma 4]{NPcusp}. Suppose that there is an element of $\pi_1 \tonde{M}$ that fixes an edge of the Bass-Serre tree. It must come from a parabolic isometry of a chamber corresponding to an endpoint of the edge. Parabolic isometries in neutered spaces fix exactly one horosphere, so if an element of $\pi_1 \tonde{M}$ fixes two different edges, it must necessarily be the identity.
\item Vertex groups are relatively hyperbolic with respect to the cusp subgroups by \cite[Theorem 5.1]{FarbRHG}. They act as convergence groups on the boundary of the $\F$-hyperbolic space to which pieces of $M$ are locally isometric.
\item Cusp subgroups are also fully quasi-convex in the sense of Dahmani \cite[Definition 1.6]{DahmaniCombConvGrps} in the fundamental groups of the pieces in which they embed. In fact, the limit set of one such cusp subgroup $H$ is just a point, on which $H$ trivially acts as a convergence group, and which is not stabilized by any non-trivial coset of $H$.
\end{itemize}
The first assertion now follows. Now, by an equivalent characterization of relative hyperbolicity given in \cite[Theorem 8.5]{DrutuSapir}, any asymptotic cone of $\pi_1 \tonde{M}$ is tree-graded with respect to the family of asymptotic cosets. The second claim follows by the bi-Lipschitz equivalence of asymptotic cones of $\pi_1 \tonde{M}$ and $\wtde{M}$ discussed earlier.
\end{proof}

We now explore the consequences of the above theorem.

\begin{cor} \label{wallsDiverge}
Different walls in $\wtde{M}$ lie at infinite Hausdorff distance from each other. More precisely, for every positive real number $R'$ there is a positive real number $D'$ depending only on $\wtde{M}$ such that the diameter of the intersection of the $R'$-neighbourhoods of two different walls is less or equal than $D'$.
\end{cor}

\begin{proof}
The proof follows from an equivalent characterization of asymptotic tree-gradedness given in \cite[Theorem 4.1]{DrutuSapir}.
\end{proof}

From the compactness of pieces it follows that any chamber is at a finite Hausdorff distance from the union of the walls adjacent to it. This, together with the preceeding corollary, immediately implies the following

\begin{cor}
Two distinct chambers are at infinite Hausdorff distance from each other.
\end{cor}

%
%

The relative hyperbolicity implies an important result on the large scale behaviour of walls in $\wtde{M}$.

\begin{lem} \label{wallsBiLipEmb}
There is a real constant $Q \geq 1$, depending only on $\wtde{M}$, such that every wall is $Q$-bi-Lipschitz embedded in $\wtde{M}$.
\end{lem}

\begin{proof}
Let $W$ be a wall, and let $p$ and $q$ be two points in $W$. Denote by $d_{\wtde{M}}$ the distance function induced by the Riemannian metric on $\wtde{M}$ and by $d_W$ the distance function on $W$ induced by the restriction of this metric to $W$. It is obvious that $d_{\wtde{M}} \tonde{p,q} \leq d_W \tonde{p,q}$.

The relative hyperbolicity of $\pi_1 \tonde{M}$ implies by \cite[Lemma 4.15]{DrutuSapir} that cusp subgroups are uniformly quasi-convex in $\pi_1 \tonde{M}$, so by \cite[Lemma III.$\Gamma$.3.5]{BrHa} there exist non-negative real numbers $L \geq 1$ and $c$ such that every cusp subgroup is $\tonde{L,c}$-quasi-isometrically embedded in $\pi_1 \tonde{M}$. The discussion at the beginning of this subsection implies that there exist non-negative real numbers $L' \geq 1$ and $c'$ such that every wall is $\tonde{L',c'}$-quasi-isometrically embedded in $\wtde{M}$. It follows that, whenever $d_M \tonde{p,q} \geq L'c'$, the inequality $d_W \tonde{p,q} \leq \tonde{L'+1} d_M \tonde{p,q}$ holds.

Suppose now that $d_{\wtde{M}}\tonde{p,q} \leq D_W$, where $D_W$ is the constant provided by Lemma \ref{DWDTW}. Then any minimizing geodesic $\gamma$ between $p$ and $q$ in $\wtde{M}$ cannot exit the union of $W$ and of the chambers adjacent to it. A chamber $C$ with the distance function induced by the restriction of the Riemannian metric is a neutered space, and the thin walls bounding $C$ are its horospheres, and are thus convex in $C$. It follows that $\gamma$ must actually stay inside $W$, and then $d_W \tonde{p,q} = d_{\wtde{M}} \tonde{p,q}$.


The only remaining case is when $D_W < d_{\wtde{M}} \tonde{p,q} < L'c'$. Let $D_M$ be the diameter of $M$. Up to the action by isometries of $\pi_1 \tonde{M}$ on $\wtde{M}$ we can suppose that both $p$ and $q$ stay in the same fixed closed ball $\mathcal{B}$ of $\wtde{M}$ of radius $D_M + L'c'$ with respect to $d_{\wtde{M}}$. This ball is compact by the completeness of $\wtde{M}$. Let $\Delta_{D_W} \subseteq \mathcal{B} \times \mathcal{B}$ be the set of pairs of points whose distance in $\wtde{M}$ is strictly less than $D_W$. Then the set $\mathcal{K} = \tonde{W \times W} \cap \tonde{\mathcal{B} \times \mathcal{B} \setminus \Delta_{D_W}}$ is also compact. Let $Q'$ be the maximum on $\mathcal{K}$ of the function $\frac{d_W}{d_{\wtde{M}}}$, which is finite because the denominator is bounded from below by $D_W$ and $\mathcal{K}$ is compact. Set finally $Q = \max \graffe{L'+1,Q'}$.
\end{proof}

\begin{lem} \label{chambersBiLipEmb}
Every chamber is $K_EQ$-bi-Lipschitz embedded in $\wtde{M}$, where $Q$ is the constant provided by the previous lemma and $K_E$ is the constant provided by Lemma \ref{biLipWalls}.
\end{lem}

\begin{proof}
Take points $p$, $q$ in the same chamber $C$, and a minimizing geodesic $\gamma$ in $\wtde{M}$ between them. Thanks to the previous lemma there exists a piecewise smooth curve $\gamma'$ between $p$ and $q$ at most $Q$ times longer and that lies entirely in the union of $C$ and of the walls adjacent to it. Now, if $\gamma''$ is a connected component of $\gamma' \cap \tonde{\wtde{M} \setminus C}$ with endpoints $x$ and $y$, by Lemma \ref{biLipWalls} $\ell \tonde{\gamma''} \leq K_E d \tonde{x,y}$, where $d$ is the path distance in a thin wall in the boundary of $C$. It follow that $C$ is $K_EQ$-bi-Lipschitz embedded in $\wtde{M}$.
\end{proof}

\section{Main results}

In this section $M$ is still a cusp-decomposable $n$-manifold. We will denote by $\G = \G \tonde{M}$ the nilpotent Lie group that covers the boundary components of the pieces in $M$. In this section we will denote by $d_R$, respectively $d_{CC}$, the Riemannian distance, respectively the Carnot-Caratheodory distance, on $\G$.

\subsection{Walls and chambers are quasi-preserved}

Here we explore the behaviour of walls and chambers of the universal covers of cusp-decomposable manifolds under quasi-isometries. The results on the behaviour of walls are an adaptation of more general facts on relatively hyperbolic groups proven in \cite{DrutuSapir} and \cite{BDM} that takes in account the particular topological and metric structure of asymptotic cones of nilpotent Lie groups.

\begin{lem} \label{AsWallsInAsWalls}
Let $M$ and $\G$ be as above. Consider an asymptotic cone $\CG \wtde{M}$ of $\wtde{M}$ and a bi-Lipschitz embedding $\phi \colon \tonde{\G,d_{CC}} \to \CG \wtde{M}$. Then the image of $f$ is an asymptotic wall in $\CG \wtde{M}$.
\end{lem}

\begin{proof}
The space $\tonde{\G,d_{CC}}$ is homeomorphic to $\R^{n-1}$. Being $n \geq 3$, it does not have cutpoints. It follows from \cite[Lemma 2.13]{DrutuSapir} that $\phi \tonde{\G}$ is entirely contained in an asymptotic wall $\mathcal{W}$. This means that $\phi$ can be thought as a bi-Lipschitz map from $\G$ to a space homeomorphic to $\R^{n-1}$. The map $\phi$ is closed because it is bi-Lipschitz. This, together with the invariance of domain Theorem, implies that $\phi \tonde{\G} = \mathcal{W}$ and $f$ is a homeomorphism with its image.
\end{proof}

\begin{lem} \label{wallsGoNearWalls}
Let $K \geq 1$ and $c \geq 0$ be real numbers. There is a positive real number $B$, depending only on $K$, $c$ and $\wtde{M}$ such that for any $\tonde{K,c}$-quasi-isometric embedding $f \colon \G \to \wtde{M}$ the image $f \tonde{\G}$ lies at a Hausdorff distance at most $B$ from a unique wall in $\wtde{M}$.
\end{lem}

\begin{proof}
Let $\UF$ be a non-principal ultrafilter on $\N$. We denote by $\tonde{e}$ the sequence of points in $\G$ constantly equal to the identity element.

Suppose that the thesis is false. For any positive integer $m$ let $f_m$ be a $\tonde{K,c}$-quasi-isometric embedding of $\G$ in $\wtde{M}$ whose image stays at Hausdorff distance greater than $m$ from any wall. Consider asymptotic cones
\[ \CG \tonde{\G, \tonde{e}, \tonde{m}_{m \in \N}, \UF} \mbox{ and } \CG \tonde{\wtde{M}, \tonde{f_m \tonde{e}}, \tonde{m}_{m \in \N}, \UF} \]
of $\G$ and $\wtde{M}$ respectively and the $K$-bi-Lipschitz embedding $\CG f$ between them induced by the sequence $\tonde{f_m}$ via Lemma \ref{QItoBL}. By Lemma \ref{AsWallsInAsWalls} the image $f \tonde{\G}$ is an asymptotic wall $\mathcal{W}$ in $\CG \wtde{M}$, which is given by a sequence of walls $\tonde{W_m}$ of $\wtde{M}$, uniquely determined up to $\UF$-a.e. equality. 

By hypothesis either for every $m$ there is a point $p'_m$ of $f_m \tonde{\G}$ at a distance greater than $m$ from $W_m$ or for every $m$ there is a point $p''_m$ of $W_m$ at a distance greater than $m$ from $f_m \tonde{\G}$. In the former case, for every positive integer $m$ we consider a geodesic $\gamma'_m \colon \quadre{0,T_m} \to \G$ between $e$ and a pre-image of $p'_m$. Let $t_m = \inf \graffe{t \in \quadre{0,T_m} | d \tonde{f_m \tonde{\gamma'_m \tonde{t}},W_m} > m}$ and let $p_m = f_m \tonde{\gamma'_m \tonde{t_m}}$. The $\UF-\lim \frac{d \tonde{p_m,f_m \tonde{e}}}{m}$ is infinite by construction.

Now we consider other asymptotic cones, namely
\[ \CG' \tonde{\G, \tonde{\gamma'_m \tonde{t_m}}, \tonde{m}, \UF} \mbox{ and } \CG' \tonde{\wtde{M}, \tonde{p_m}, \tonde{m}, \UF} \]
with the map $\CG' f$ between them induced by $\tonde{f_m}$. Consider furthermore the $\tonde{K,c}$-quasi-geodesics $\gamma_m \colon \quadre{0,t_m} \to \wtde{M}$ given by $\gamma_m \tonde{t} = f \circ \gamma'_m \tonde{t_m-t}$. By construction and infiniteness of the above $\UF$-limit, for any non-negative real number $t$ the point $\gamma \tonde{t}$ given by the class of the sequence $\quadre{\gamma_m \tonde{mt}}$ is well-defined and the resulting curve $t \mapsto \gamma \tonde{t}$ defined on the non-negative real half line is $K$-bi-Lipschitz. Also note that $\gamma \tonde{0}$ is precisely $\quadre{p_m}$ and that all the points in the image of $\gamma$ lie in $\CG' f \tonde{\CG' \G}$ and are at a distance at most $1$ from the asymptotic wall $\mathcal{W}'$ given by the same sequence $\tonde{W_m}$ as above; in particular, the point $\gamma \tonde{0}$ is at distance exactly $1$.

However, the image of $\CG' \G$ via $\CG' f$ in $\CG' \wtde{M}$ is still an asymptotic wall by Lemma \ref{AsWallsInAsWalls}, different from $\mathcal{W}'$ because this image contains $\gamma \tonde{0}$ and $\mathcal{W}'$ does not. The image of $\gamma$ provides us, in an asymptotic cone of $\wtde{M}$, a set of infinite diameter that lies in the intersection between the $2$-neighbourhood of an asymptotic wall and another asymptotic wall. This is in contradiction with the tree-gradedness of $\CG' \wtde{M}$.

An analogous argument with the roles of the sequences $\tonde{f_m \tonde{\G}}$ and $\tonde{W_m}$ switched holds if we consider the second possibility. The only difference is that in this case $\gamma'_m$ is a geodesic in $\wtde{M}$ between $f_m \tonde{e}$ and $p''_m$.
\end{proof}

We can sum up what we showed so far in the following

\begin{cor} \label{wallsBij}
Let $M_1$ and $M_2$ be two cusp-decomposable manifolds of dimensions $n_1$ and $n_2$ respectively, and let $f$ be a quasi-isometry between their universal covers $\wtde{M_1}$ and $\wtde{M_2}$. Then $n_1=n_2$, the pieces of $M_1$ and $M_2$ are locally isometric to the same symmetric geometry and there is a bijection between walls in $\wtde{M_1}$ and $\wtde{M_2}$ that takes a wall $W$ in $\wtde{M_1}$ to the unique wall in $\wtde{M_2}$ such that $f \tonde{W}$ is at a finite Hausdorff distance from it.
\end{cor}

\begin{proof}
Suppose without loss of generality that $n_1 \geq n_2$. A wall in $\wtde{M_1}$ is at a finite Hausdorff distance from any its thin wall, which is the image of a bi-Lipschitz embedding $\iota$ of $\tonde{\G \tonde{M_1},d_R}$. Consider some asymptotic cones $\CG \G \tonde{M_1}$, $\CG \wtde{M_1}$ and $\CG \wtde{M_2}$ of $\G \tonde{M_1}$, $\wtde{M_1}$ and $\wtde{M_2}$ respectively. Remember from Lemma \ref{asConesNilp} that $\CG \G \tonde{M_1}$ is bi-Lipschitz homeomorphic to $\tonde{\G \tonde{M_1},d_{CC}}$, and is thus in particular homeomorphic $\R^{{n_1}-1}$ by Lemma \ref{dCCsameTop}. By \cite[Lemma 2.13]{DrutuSapir} the image of the bi-Lipschitz embedding $\CG f \circ \CG \iota$ of $\tonde{\G\tonde{M_1},d_{CC}}$ in $\CG \wtde{M_2}$ must lie in an asymptotic wall $\barr{W}$ of $\CG \wtde{M_2}$. Analogously, the asymptotic wall $\barr{W}$ is homeomorphic to $\R^{n_2-1}$. From the invariance of domain it follows that $n_1 = n_2$ and that $\CG f \circ \CG \iota \tonde{\CG \G \tonde{M_1}}$ is an open subset of $\barr{W}$. The map $\CG f \circ \CG \iota$ then provides a bi-Lipschitz homeomorphism between $\tonde{\G \tonde{M_1},d_{CC}}$ and an open subset of $\barr{W}$, which is bi-Lipschitz homeomorphic to $\tonde{\G \tonde{M_2},d_{CC}}$. Lemma \ref{DistNilpGroups} implies that $\G \tonde{M_1}$ and $\G \tonde{M_2}$ are isomorphic Lie groups, and hence that pieces of $M_1$ and $M_2$ are locally isometric to the same symmetric geometry.

The good definition and injectivity of the function induced on walls follows now from Lemma \ref{wallsGoNearWalls}, the fact that different walls lie at infinite Hausdorff distance and the definition of quasi-isometry. In order to get bijectivity it is sufficient to apply the same argument to a quasi-inverse of $f$.
\end{proof}

We now consider the behaviour of quasi-isometries on chambers. First we need some additional results on the geometry of $\wtde{M}$.

\begin{defin}
Let $W_1$ and $W_2$ be two distinct walls of $\wtde{M}$. We say that a third wall $W$ \emph{separates} $W_1$ and $W_2$ if $W_1$ and $W_2$ are contained in two different connected components of $\wtde{M} \setminus W$. Equivalently, the egdes corresponding to $W_1$ and $W_2$ lie in different connected components of the Bass-Serre tree of $\wtde{M}$ with the edge corresponding to $W$ removed; or any path from a point of $W_1$ to a point of $W_2$ intersects $W$.
\end{defin}

Now we want to characterize the set of walls adjacent to a certain chamber of $\wtde{M}$ in terms of separation.

\begin{lem} \label{wallsOfSameChamber}
Let $\Omega$ be a collection of walls of $\wtde{M}$. Then there exists a chamber $C$ that is adjacent to all the walls in $\Omega$ if and only if for any pair of distinct walls in $\Omega$ no other wall in $\wtde{M}$ separates them.
\end{lem}

\begin{proof}
If two walls $W_1$ and $W_2$ are adjacent to the same chamber $C$, and $W$ is a third wall, then $W_1$ and $W_2$ lie in the same connected component of $\wtde{M} \setminus W$ as $C$. This proves the ``only if'' part of the assertion.

Suppose now that no pair of walls in $\Omega$ can be separated by some wall in $\wtde M$, and let $W'$ be a wall in $\Omega$. By hypothesis, all the other walls in $\Omega$ must belong to the same connected component of $\wtde{M} \setminus W'$. Let then $C'$ be the chamber adjacent to $W'$ in the same connected component of $\wtde{M} \setminus W'$ as the other walls in $\Omega$. We claim that all these other walls are adjacent to $C'$. Suppose that there is a wall $W''$ which is not adjacent to $C'$. Consider the chamber $C''$ adjacent to $W''$ that lies in the same connected component of $\wtde{M} \setminus W''$ as $W'$ and $C'$. Then $C' \neq C''$, and any wall corresponding to an edge in the unique injective path between vertices corresponding to $C'$ and $C''$ in the Bass-Serre tree of $\wtde{M}$ separates $W'$ and $W''$.
\end{proof}

\begin{lem} \label{connComplement}
Let $\mathcal{N}$ be a neutered space of dimension at least $3$, and $R$ a positive real number. Then the complement $\mathcal{S}$ in $\mathcal{N}$ of the open $R$-neighbourhood of a boundary horosphere is path-connected.
\end{lem}

\begin{proof}
The neutered space $\mathcal{N}$ is the complement of a collection $\tonde{\mathcal{B}_i}$ of pairwise non-intersecting open horoballs in the symmetric space $Y$. By definition, for every $i$ there exists a Busemann function $b_i$ such that $\mathcal{B}_i$ is precisely the set $\graffe{b_i<0}$. The subspace $\mathcal{S}$ is obtained by taking in $Y$ the complement of all $\mathcal{B}_i$'s and of the open $R$-neighbourhood of a distinguished horoball $\mathcal{B}_{i_0}$, which is precisely the set $\graffe{b_{i_0}<R}$.

The space $\mathcal{N}$ is path-connected. This means that for every $p$ in $\mathcal{N}$ with $b_{i_0} \tonde{p} \geq R$ there is a path $\alpha \colon \quadre{0,1} \to \mathcal{N}$ joining it with a point in the set $\graffe{b_{i_0} = 0}$. If $t_p$ is the minimum of the non-empty closed set $\tonde{b_{i_0} \circ \alpha}^{-1} \tonde{\graffe{R}}$, then $\restr{\alpha}{\quadre{0,t_p}}$ is a path between $p$ and a point in the set $\graffe{b_{i_0} = R}$ entirely lying in $\mathcal{S}$, so it only remains to prove that the subspace $\mathcal{N} \cap \graffe{b_{i_0} = R}$ of $Y$ is path-connected.

Let $\mathcal{O}$ be the horosphere $\graffe{b_{i_0} = R}$ in $Y$. For every index $i$, let $\omega_i$ be the point at infinity of $Y$ to which the Busemann function $b_i$ is relative. For every $i \neq i_0$ such that $\mathcal{B}_i \cap \mathcal{O}$ is non-empty consider the unique geodesic line $\gamma_i$ in $Y$ such that one of its half-lines is in the class of $\omega_{i_0}$ and the other in the class of $\omega_{i}$, and let $p_i$ be the intersection of the image of $\gamma_i$ with $\mathcal{O}$. Note that $p_i \in \mathcal{B}_i \cap \mathcal{O}$. Using the gradient flow of the functions $b_i$ restricted to $\mathcal{O}$ one can prove that $\mathcal{N} \cap \mathcal{O}$ is a deformation retract of $\mathcal{O}$ with the points $p_i$ removed. This latter space is a discrete collection of points by construction, and thus its complement is path-connected being the dimension of $Y$ at least $3$ and thus the dimension of $\mathcal{O}$ at least $2$. But then the homotopically equivalent space $\mathcal{N} \cap \mathcal{O}$ is also path-connected.
\end{proof}

Now we can prove the following

\begin{lem} \label{wallsInSameChamberRemainInSameChamber}
Take two cusp-decomposable manifolds $M_1$ and $M_2$. Let $f$ be a $\tonde{K,c}$-quasi-isometry between their universal covers $\wtde{M_1}$ and $\wtde{M_2}$ and let $\barr{f}$ be the associated bijection between walls given by Corollary \ref{wallsBij}. Then the collection of walls adjacent to a chamber $C$ in $\wtde{M_1}$ is sent by $\barr{f}$ to the collection of walls adjacent to one chamber in $\wtde{M_2}$.
\end{lem}

\begin{proof}
By Lemma \ref{wallsOfSameChamber} it is enough to prove that for any pair of walls $W_1$ and $W_2$ adjacent to $C$ the walls $\barr{f} \tonde{W_1}$ and $\barr{f} \tonde{W_2}$ are not separated by any wall in $\wtde{M_2}$. A wall in $\wtde{M_2}$ is necessarily $\barr{f} \tonde{W}$ for some wall $W$ in $\wtde{M_1}$. We will construct a continuous path $\gamma$ with endpoints on $W_1$ and $W_2$ that always stays at a distance at least $D=K\tonde{K+2c+B}+1$ away from $W$, where $B$ is the constant given by Lemma \ref{wallsGoNearWalls}. This way the image path $f \circ \gamma$ stays at a distance strictly greater than $K+c+B$ from $f \tonde{W}$ and then at a distance strictly greater than $K+c$ from $\barr{f} \tonde{W}$. However, if $\barr{f} \tonde{W}$ separated $\barr{f} \tonde{W_1}$ and $\barr{f} \tonde{W_2}$, at least one point of $f \circ \gamma$ should stay at a distance strictly less than $K+c$ from $\barr{f} \tonde{W}$.

There are two possible mutual positions of $W_1$, $W_2$ and $W$:
\begin{itemize}
\item None of the three walls separates the other two;
\item The wall $W_1$ separates $W$ and $W_2$, and the analogous situation with $1$ and $2$ inverted.
\end{itemize}
In fact it is not possible that $W$ separates the other two walls by construction.

In the first case, let $W'$ be the wall of $C$ which is either $W$ itself or which leaves $C$ and $W$ in different connected components of $\wtde{M}$ when removed. The chamber $C$ is bi-Lipschitz embedded in $\wtde{M_1}$, and the restriction of the metric tensor of $\wtde{M_1}$ to $C$ makes it isometric to a neutered space with distance function $d_C$. This means that there is a positive real number $R$ such that the $R$-neighbourhood $W'_R$ of $W'$ with respect to $d_C$ contains the $D$-neighbourhood of $W$ with respect to the metric of the whole $\wtde{M_1}$. By the previous lemma, the complement in $C$ of $W'_R$ is connected, and no wall is contained in the $R$-neighbourhood of another, so there are points of $W_1$ and $W_2$ connected by a piecewise smooth path staying outside this neighbourhood.

In the second case, an easy consequence of Corollary \ref{wallsDiverge} and of the bi-Lipschitz embedding of $C$ in $\wtde{M_1}$ is that the intersection of $C$ with the $D$-neighbourhood of $W$ in $\wtde{M_1}$ is contained in a $d_C$-ball centered on the thin wall separating $C$ and $W_1$. Take a third wall $W_3$ in $C$ different from $W_1$ and $W_2$. The above mentioned ball is contained in an $R$-neighbourhood of $W_3$ with respect to $d_C$ for $R$ large enough, and the complement in $C$ of this neighbourhood is still connected by the previous lemma, which allows us to conclude in the same way as in the first case.
\end{proof}

%
%
%

Finally we are ready to study the behaviour of chambers under quasi-isometries.

\begin{teo} \label{chambersQuasiPreserved}
Let $M_1$ and $M_2$ be two cusp-decomposable manifolds and $f$ be a $\tonde{K,c}$-quasi-isometry between their universal covers $\wtde{M_1}$ and $\wtde{M_2}$. Then the image via $f$ of a chamber of $\wtde{M_1}$ is at a finite Hausdorff distance from a unique chamber in $\wtde{M_2}$.
\end{teo}

\begin{proof}
Let $C$ be a chamber in $\wtde{M_1}$ and let $\barr{f}$ be the bijection between the walls of the universal covers guaranteed by Corollary \ref{wallsBij}. Lemma \ref{wallsInSameChamberRemainInSameChamber} implies that the set of walls adjacent to $C$ is mapped by $\barr{f}$ to the set of walls adjacent to some chamber $C'$ of $\wtde{M_2}$. We claim that all walls of $C'$ are reached. Indeed, by applying the same argument to a quasi-inverse of $f$ any wall adjacent to $C'$ should be taken back to a wall adjacent to some chamber $C''$ in $\wtde{M_1}$. Now necessarily $C = C''$ because different chambers have infinite Hausdorff distance. The fact that $f \tonde{C}$ lies at a finite Hausdorff distance from $C'$ follows from the observation that a chamber lies at a finite Hausdorff distance from the union of its walls. There are no other chambers in $\wtde{M}$ lying at a finite Hausdorff distance from $f \tonde{C}$ because different chambers are at infinite Hausdorff distance from each other.
\end{proof}

Similarly to Corollary \ref{wallsBij}, we get the following

\begin{cor}\label{PresChambers}
There is a bijection between chambers in $\wtde{M_1}$ and $\wtde{M_2}$ that takes a chamber $C$ in $\wtde{M_1}$ to the unique chamber in $\wtde{M_2}$ such that $f \tonde{C}$ is at a finite Hausdorff distance from it.
\end{cor}

Also, the proof of Lemma \ref{chambersQuasiPreserved} implies that the images of a chamber and of a wall adjacent to it via the bijections induced by a quasi-isometry form an adjacent pair of chamber and wall. This implies immediately the following corollaries.

\begin{cor}
The bijection induced by a quasi-isometry between walls in $\wtde{M_1}$ and walls in $\wtde{M_2}$ takes boundary walls to boundary walls.
\end{cor}

\begin{cor} \label{treeIso}
The bijections induced by a quasi-isometry on vertices and edges of the Bass-Serre trees of $\wtde{M_1}$ and $\wtde{M_2}$ define a tree isomorphism.
\end{cor}

\begin{cor} \label{sepPreserved}
The bijection $\barr{f}$ induced by a quasi-isometry $f$ between walls in $\wtde{M_1}$ and walls in $\wtde{M_2}$ preserves separation, i.e. for any triple $W$, $W_1$, $W_2$ of distinct walls in $\wtde{M_1}$ the wall $\barr{f} \tonde{W}$ separates $\barr{f} \tonde{W_1}$ and $\barr{f} \tonde{W_2}$ if and only if $W$ separates $W_1$ and $W_2$.
\end{cor}

We restate these results in terms of groups.

\begin{teo} \label{HausGroupPieces}
Consider two cusp-decomposable manifolds $M_1$ and $M_2$ and let $f \colon \pi_1 \tonde{M_1} \to \pi_1 \tonde{M_2}$ be a quasi-isometry. Let $P$ be a piece in $M_1$, and $\pi_1 \tonde{P}$ the image of its fundamental group in $\pi_1 \tonde{M_1}$. Then $f \tonde{\pi_1 \tonde{P}}$ is at a finite Hausdorff distance from a conjugate of the fundamental group of a piece in $\pi_1 \tonde{M_2}$.
\end{teo}

\begin{proof}
Let $P_{1}$, \dots, $P_{m}$ be the pieces of $M_2$. Take basepoints $x_1$ in $P$ and $x_2$ in $M_2$. The set of left translates of $x_1$ via $\pi_1 \tonde{P}$ is at a finite Hausdorff distance from a chamber $C$ in the universal cover $\wtde{M_1}$ of $M_1$. Theorem \ref{chambersQuasiPreserved} tells us that the quasi-isometry induced by $f$ between the universal covers $\wtde{M_1}$ and $\wtde{M_2}$ via the Milnor-\v{S}varc lemma takes $C$ at a finite Hausdorff distance from a chamber $C'$ in $\wtde{M_2}$. The chamber $C'$ is at a finite Hausdorff distance from the set of left translates of $x_2$ by some left coset $g\pi_1 \tonde{P_i}$ of the fundamental group of a piece in $M_2$. Consider some word metric on $\pi_1 \tonde{M_2}$ and note that for every $h$ in $\pi_1 \tonde{M_2}$ the distance between elements $h$ and $hg^{-1}$ is always equal to the distance of $g^{-1}$ from the identity. This means that $f \tonde{\pi_1 \tonde{P}}$ is at a finite Hausdorff distance from $g \pi_1 \tonde{P_i} g^{-1}$.
\end{proof}

\subsection{Quasi-isometric rigidity of groups}

We are finally going to explore the structure of groups that are quasi-isometric to the fundamental group of a cusp-decomposable manifold. Let $M$ be a cusp-decomposable manifold and let $\Gamma$ be a finitely generated group quasi-isometric to its universal cover $\wtde{M}$. By Lemma \ref{QItoQIaction}, we have a cobounded quasi-action of $\Gamma$ on $\wtde{M}$. By Corollary \ref{treeIso}, this action induces an action by automorphisms on the Bass-Serre $T$ of the decomposition in chambers of $\wtde{M}$. We would like to apply the fundamental theorem of the Bass-Serre theory to $\Gamma$. In order to do this, the action of $\Gamma$ on the tree must be without edge inversions.

\begin{lem} \label{noEdgeInv}
Either the group $\Gamma$ itself acts on $T$ without edge inversions, or it has an index two subgroup that acts without edge inversions.
\end{lem}

\begin{proof}
A tree is a bipartite graph. If a group acts on a bipartite graph by automorphisms, then the subgroup that sends each subset of the partition into itself is of index at most two, and clearly cannot invert edges. 
\end{proof}

Therefore, up to passing to an index $2$ subgroup of $\Gamma$, which does not change the virtual isomorphism and quasi-isometry class, we may assume that the action of $\Gamma$ on $T$ is without edge inversions.

\begin{lem}
There is only a finite number of orbits of edges under the action of $\Gamma$ on $T$.
\end{lem}

\begin{proof}
This fact is standard, see \cite[Lemma 10.3]{FLSgraph} for example, so we only give here an idea of the proof. Fix a basepoint $x_0$ in $\wtde{M}$ and consider a wall $W$. The quasi-action $\rho$ of $\Gamma$ on $\wtde{M}$ is cobounded, so there is a point $\rho \tonde{g} \tonde{x_0}$ in the orbit of $x_0$, where $g$ is an element of $\Gamma$, at a distance from $W$ bounded by some constant depending only on the quasi-action. This means that the wall at a finite Hausdorff distance from $\rho \tonde{g^{-1}} \tonde{W}$ is at a bounded distance from $x_0$. By Corollary \ref{locFinWalls} the walls at a distance smaller than a fixed positive real number from $x_0$ are finite, so we have a finite number of representatives for the orbits of the action on walls.
\end{proof}

\begin{cor}
The quotient of $T$ under the action of $\Gamma$ is a finite graph.
\end{cor}

We then describe the structure of edge and vertex stabilizers under the action of $\Gamma$ on $T$.

\begin{lem} \label{geoStructStab}
Every edge stabilizer is a finitely generated group quasi-isometric to $\G \tonde{M}$ endowed with a left invariant Riemannian metric. Every vertex stabilizer is a finitely generated group quasi-isometric to the chamber corresponding to the vertex with the induced Riemannian metric.
\end{lem}

\begin{proof}
This lemma is also standard, see \cite[Lemma 10.4]{FLSgraph} for example. We will give an idea of the proof for the walls; the one for chambers is analogous. Consider a wall $W$ and a basepoint $x_0$ in it. Denote by $\Gamma_W$ the stabilizer of the edge corresponding to $W$. By Lemma \ref{wallsGoNearWalls} we know that the orbit via $\Gamma_W$ of $x_0$ is contained in the $B$-neighbourhood of $W$ which is quasi-isometric to $W$ as a consequence of Lemma \ref{wallsBiLipEmb}. It remains to show that every point of this neighbourhood, and hence of $W$, is at a bounded distance from this orbit; this implies that there is a quasi-action of $\Gamma_W$ on $W$. The thesis then follows from a version of the Milnor-\v{S}varc lemma for quasi-actions \cite[Lemma 1.4]{FLSgraph}.

Let $x$ be a point of $W$. By hypothesis, there is a point $\rho \tonde{g} \tonde{x_0}$ of the orbit of $x_0$ via the whole $\Gamma$ at a bounded distance from $x$. Then the point $\rho \tonde{g^{-1}} \tonde{x}$ is at a bounded distance from $x_0$ and stays near a wall $\barr{W}$ in the orbit of $W$ via the action of $\Gamma$ on walls. As in the proof of the previous lemma, the wall $\barr{W}$ can be one among a finite number of walls $W_1$, \dots, $W_m$, chosen independently of $W$. Suppose $\barr{W} = W_i$. Let $g_1$, \dots, $g_m$ be elements of $\Gamma$ that bring $W$ in $W_1$, \dots, $W_m$ respectively via the action on walls. The points $\rho \tonde{g_1} \tonde{x_0}$, \dots, $\rho \tonde{g_m} \tonde{x_0}$ obviously stay at a bounded distance from $x_0$, and thus $gg_i$ is in $\Gamma_W$ and the point $\rho \tonde{gg_i} \tonde{x_0}$ is at a bounded distance from $x$.
\end{proof}

We now will prove that finitely generated groups quasi-isometric to the fundamental group of a cusp-decomposable manifold are orbifold fundamental groups of cusp-decomposable orbifolds. We first define these objects.

\begin{defin}
A \emph{cusp-decomposable orbifold} is obtained by gluing along affine homeomorphisms of some pairs of boundary components a finite number of pieces; each piece is obtained by a complete non-compact finite volume locally symmetric orbifold as in Definition \ref{lsOrb} by troncating its cusps.
\end{defin}

We will need the following

\begin{lem}
The orbifold fundamental group of a cusp-decomposable orbifold is the fundamental group of a finite graph of groups where each vertex corresponds to a piece, each edge corresponds to a pair of glued boundary components, every vertex group is the orbifold fundamental group of the piece corresponding to the vertex, every edge group is the orbifold fundamental group of a glued boundary component, and the maps between edge and vertex groups are the natural inclusions.
\end{lem}

\begin{proof}
The proof is the same as in the manifold case. The only difference is in using the Van Kampen theorem for orbifold fundamental groups instead of the classical one. 
\end{proof}

Now we are ready to prove the main theorem of this paper, which we restate in the most complete form.

\begin{teo} \label{main}
Let $\Gamma$ be a finitely generated group quasi-isometric to the universal cover of a cusp-decomposable manifold $M$. Denote by $Y$ the symmetric space to which pieces in $M$ are locally isometric. Then $\Gamma$ is virtually isomorphic to the fundamental group of a cusp-decomposable orbifold. More precisely, $\Gamma$ itself, or a subgroup of index $2$, has a finite normal subgroup $F$ which is the intersection of all the stabilizers of vertices and edges of the Bass-Serre tree $T$ of $M$ under the action of $\Gamma$ given by Lemma \ref{QItoQIaction} and Corollary \ref{treeIso}. The quotient $\Gamma/F$ is the orbifold fundamental group of an orbifold constructed by gluing pieces as follows:
\begin{enumerate}
\item The pieces are in bijection with the vertices of the quotient graph of $T$ by the action of $\Gamma$;
\item Let $v$ be a vertex of the quotient graph, and let $\wtde{v}$ be any of its preimages in $T$. Denote by $C_{\wtde{v}}$ the corresponding chamber in $\wtde{M}$ and by $\Gamma_{\wtde{v}}$ the stabilizer of $\wtde{v}$ by the action of $\Gamma$ on $T$. Then the piece corresponding to $v$ is isometric to the quotient of a suitable neutered space $\mathcal{N}_{\wtde{v}}$, whose boundary horospheres are in natural bijection with walls adjacent to $C_{\wtde{v}}$, by a properly discontinuous cocompact isometric action of a group isomorphic to $\Gamma_{\wtde{v}}/F$. Furthermore, this group shares a finite index subgroup with the fundamental group of the piece in $M$ covered by $C_{\wtde{v}}$.
\item The pieces are glued along affine homeomorphisms of boundary components. The glued pairs are in bijection with the edges of the quotient graph.
\end{enumerate}
\end{teo}

\begin{proof}
By Lemma \ref{noEdgeInv}, up to index $2$ subgroups we may assume that $\Gamma$ acts on $T$ without edge inversions. This allows us to conclude that $\Gamma$ is the fundamental group of a graph of groups whose supporting graph is the quotient of $T$ under the action of $\Gamma$, the vertex group relative to a vertex $v$ is isomorphic to the stabilizer of any preimage of $v$ in $T$ and the edge group relative to an edge $e$ is the stabilizer of any preimage of $e$ in $T$. In what follows we will give a geometric interpretation of these stabilizers.

Let $\wtde{v}$ be any vertex in the Bass-Serre tree of $M$ and $C_{\wtde{v}}$ the corresponding chamber in $\wtde{M}$. Lemmas \ref{geoStructStab} and \ref{QItoQIaction} tell us that the stabilizer $\Gamma_{\wtde{v}}$ in $\Gamma$ of ${\wtde{v}}$ has a quasi-action $\rho$ on $C_{\wtde{v}}$. 
The paper by Schwarz \cite[\S 10.4]{SchwRank1} describes very well this quasi-action, which is as follows.

The space $C_{\wtde{v}}$ endowed with the Riemannian metric induced by $\wtde{M}$ is a neutered space in $Y$; we then consider $C_{\wtde{v}}$ as a subspace of $Y$. In $Y$ there is another neutered space $\mathcal{N}_{\wtde{v}}$ containing $C_{\wtde{v}}$ such that for every horosphere of $C_{\wtde{v}}$ there is a unique horosphere of $\mathcal{N}_{\wtde{v}}$ with the same point at infinity, and viceversa. 
Furthermore there is a positive real constant $B_{\wtde{v}}$ such that
\begin{itemize}
\item Two horospheres of $C_{\wtde{v}}$ and $\mathcal{N}_{\wtde{v}}$ with the same point at infinity lie at a Hausdorff distance at most $B_{\wtde{v}}$;
\item For every $g$ in $\Gamma_{\wtde{v}}$ there is an isometry $\barr{g}$ of $\mathcal{N}_{\wtde{v}}$ such that for every $x$ in $C_{\wtde{v}}$ the inequality $d\tonde{\rho\tonde{g}\tonde{x}, \barr{g}\tonde{x}} \leq B_{\wtde{v}}$ holds.
\end{itemize}
The collection of all the isometries corresponding to elements of $\Gamma_{\wtde{v}}$ forms a group of isometries of $\mathcal{N}_{\wtde{v}}$ acting properly discontinuously and cocompactly. The elements $g$ such that $\barr{g}$ is the identity of $\mathcal{N}_{\wtde{v}}$ form \emph{the} maximal finite normal subgroup $F_{\wtde{v}}$ of $\Gamma_{\wtde{v}}$. Then the quotient $\Gamma_{\wtde{v}}/F_{\wtde{v}}$ acts by isometries on $\mathcal{N}_{\wtde{v}}$ as well. This quotient shares a finite index subgroup with the fundamental group of the piece covered by $C_{\wtde{v}}$.

Let now $\wtde{w}$ be any vertex of $T$ which is in the orbit of $\wtde{v}$ under the action of $\Gamma$. Fix an element $g_{vw}$ of $\Gamma$ that takes $\wtde{v}$ in $\wtde{w}$. It induces a quasi-isometry between the corresponding chambers $C_{\wtde{v}}$ and $C_{\wtde{w}}$. These chambers are subspaces of $Y$. From \cite[Lemma 6.1]{SchwRank1} there is a unique isometry $\barr{g_{vw}}$ of $Y$ into itself such that the image of a point of $C_{\wtde{v}}$ via $\restr{\barr{g_{vw}}}{C_{\wtde{v}}}$ lies at a bounded distance from the image of the same point via the quasi-isometry induced by $g_{vw}$. This allows us to define the neutered space $\mathcal{N}_{\wtde{w}}$ as $\barr{g_{vw}} \tonde{\mathcal{N}_{\wtde{v}}}$. The stabilizer $\Gamma_{\wtde{w}}$ of $\wtde{w}$ is the conjugate of $\Gamma_{\wtde{v}}$ via $g_{vw}$, and the isometric action of $\Gamma_{\wtde{w}}$ on $\mathcal{N}_{\wtde{w}}$ is defined via this conjugation. This construction is equivariant with respect to the bijections between horospheres in $\mathcal{N}_{\wtde{v}}$ and $C_{\wtde{v}}$ and that between horospheres in $\mathcal{N}_{\wtde{w}}$ and $C_{\wtde{w}}$ respectively.

Now we can describe edge stabilizers. Let $\wtde{e}$ be an edge of $T$. Its stabilizer $\Gamma_{\wtde{e}}$ is clearly a subgroup of every vertex stablizer $\Gamma_{\wtde{v'}}$ corresponding to an endpoint $\wtde{v'}$ of $\wtde{e}$. On the other side, by the above description a subgroup of $\Gamma_{\wtde{v'}}$ is the stabilizer of a horosphere $\mathcal{O}$ of $\mathcal{N}_{\wtde{v}}$ if and only if in the action on $T$ it is the stabilizer of the edge corresponding to the wall adjacent to $C_{\wtde{v'}}$ corresponding to $\mathcal{O}$.

If $\wtde{e}$ is an edge that has $\wtde{v'}$ as an endpoint, from the previous discussion it follows that its stabilizer $\Gamma_{\wtde{e}}$ contains the finite group $F_{\wtde{v'}}$ and that $\Gamma_{\wtde{e}}/F_{\wtde{v}}$ acts properly discontinuously, cocompactly and faithfully, but not necessarily freely, by isometries on $\G \tonde{M}$. By \cite[Theorem 4.2]{Dekimpe} $\Gamma_{\wtde{e}}/F_{\wtde{v}}$ does not have non-trivial finite normal subgroups, and thus, if $\wtde{v'}$ is the other endpoint of $\wtde{e}$, then necessarily $F_{\wtde{v'}} \subseteq F_{\wtde{v}} \subseteq \Gamma_{\wtde{e}}$. By switching the roles of ${\wtde{v}}$ and ${\wtde{v'}}$ we get that $F_{\wtde{v}} = F_{\wtde{v'}}$. From the connectedness of $T$ it follows that $\Gamma$ has a finite subgroup $F$, that is contained in the intersection of all the stabilizers of edges and vertices, and such that $F_{\wtde{v}} = \Gamma_{\wtde{v}} \cap F$. Viceversa, if an element of $\Gamma$ stabilizes all the vertices, by the above discussion it acts trivially on every $\mathcal{N}_{\wtde{v}}$ and is thus in $F$. So $F$ is precisely the intersection of all vertex stabilizers, and is then normal in $\Gamma$ because it is the kernel of the action on $T$. In particular, also $\Gamma/F$ acts on $T$.


Now we want to prove that $\Gamma/F$ is the orbifold fundamental group of a cusp decomposable orbifold as in the statement of the Theorem. We already described the pieces and now we describe the gluings. Let $e$ be an edge of the quotient graph of $T$ by the action of $\Gamma$. Choose a preimage $\wtde{e}$ of $e$ in $T$. Let $\Gamma_{\wtde{e}}$ the stabilizer of $\wtde{e}$, and let $\wtde{v'}$ and $\wtde{v''}$ be the endpoints of $\wtde{e}$. Denote by $v'$ and $v''$ their (possibly coincident) projections to the quotient graph. Consider the boundary component of the piece corresponding to $v'$, respectively $v''$, that is covered by the horosphere in $\mathcal{N}_{\wtde{v'}}$, respectively $\mathcal{N}_{\wtde{v''}}$, 
associated to the edge $\wtde{e}$. Each of these horospheres is isometric to $\G \tonde{M}$ with some left-invariant Riemannian metric. The two boundary components are quotients of the relative horospheres by a properly discontinuous cocompact isometric action of $\Gamma_{\wtde{e}}/F$. We choose an affine homeomorphism between the two horospheres that conjugates the actions of $\Gamma_{\wtde{e}}/F$ on them as in Lemma \ref{isoFromAff}, and glue the pieces corresponding to $v'$ and $v''$ along the projection of this affine homeomorphism to the boundary components.

\end{proof}

One can easily adapt the arguments of this work to prove the following

\begin{teo} \label{corMain}
The orbifold fundamental groups of cusp-decomposable orbifolds form a quasi-isometrically rigid class, i.e. every finitely generated group quasi-isometric to an orbifold fundamental group of a cusp-decomposable orbifold is actually virtually isomorphic to the fundamental group of a, maybe different, cusp-decomposable orbifold.
\end{teo}

%
%
%
%
%

\addcontentsline{toc}{section}{Bibliography}
\bibliographystyle{alpha}
\bibliography{Dottorato}

\newcommand{\etalchar}[1]{$^{#1}$}
\begin{thebibliography}{BDM09}

\bibitem[BDM09]{BDM}
Jason Behrstock, Cornelia Dru{\c{t}}u, and Lee Mosher.
\newblock Thick metric spaces, relative hyperbolicity, and quasi-isometric
  rigidity.
\newblock {\em Mathematische Annalen}, 344(3):543--595, 2009.

\bibitem[BGS85]{BGS}
Werner Ballmann, Mikhael Gromov, and Viktor Schroeder.
\newblock {\em Manifolds of Nonpositive Curvature}.
\newblock Birkhaeuser, Boston, 1985.

\bibitem[BH99]{BrHa}
Martin~Robert Bridson and Andr\'e Haefliger.
\newblock {\em Metric Spaces of Non-Positive Curvature}, volume 319 of {\em
  Grundlehren der Mathematischen Wissenschaften}.
\newblock Springer-Verlag, 1999.

\bibitem[Dah03]{DahmaniCombConvGrps}
Fran{\c{c}}ois Dahmani.
\newblock Combination of convergence groups.
\newblock {\em Geometry \& Topology}, 7(2):933--963, 2003.

\bibitem[Dek16]{Dekimpe}
Karel Dekimpe.
\newblock A users' guide to infra-nilmanifolds and almost-{B}ieberbach groups.
\newblock {\em arXiv preprint arXiv:1603.07654}, 2016.

\bibitem[DK15]{DKlecGGT}
Cornelia Drutu and M~Kapovich.
\newblock Lectures on geometric group theory.
\newblock {\em URL https://www. math. ucdavis. edu/\~{}
  kapovich/EPR/kapovich\_drutu. pdf.(as of 29/04/2015) Book in preparation},
  2015.

\bibitem[DS{\etalchar{+}}05]{DrutuSapir}
Cornelia Dru{\c{t}}u, Mark Sapir, et~al.
\newblock Tree-graded spaces and asymptotic cones of groups.
\newblock {\em Topology}, 44(5):959--1058, 2005.

\bibitem[Far98]{FarbRHG}
Benson Farb.
\newblock Relatively hyperbolic groups.
\newblock {\em Geometric and functional analysis}, 8(5):810--840, 1998.

\bibitem[FLS15]{FLSgraph}
Roberto Frigerio, Jean-Fran{\c{c}}ois Lafont, and Alessandro Sisto.
\newblock Rigidity of high dimensional graph manifolds.
\newblock {\em Astérisque}, 372, 2015.

\bibitem[Fri16]{FrigSurvey}
Roberto Frigerio.
\newblock Quasi-isometric rigidity of piecewise geometric manifolds.
\newblock {\em arXiv preprint arXiv:1602.02603v2}, 2016.

\bibitem[Gro81]{GroNilp}
Michael Gromov.
\newblock Groups of polynomial growth and expanding maps (with an appendix by
  {Jacques Tits}).
\newblock {\em Publications Math{\'e}matiques de l'Institut des Hautes
  {\'E}tudes Scientifiques}, 53:53--78, 1981.

\bibitem[KL97]{KaLeeb3manQI}
Michael Kapovich and Bernhard Leeb.
\newblock Quasi-isometries preserve the geometric decomposition of {Haken}
  manifolds.
\newblock {\em Inventiones mathematicae}, 128(2):393--416, 1997.

\bibitem[M{\etalchar{+}}85]{MitchellCC}
John Mitchell et~al.
\newblock On {Carnot-Carath{\'e}odory} metrics.
\newblock {\em Journal of Differential Geometry}, 21(1):35--45, 1985.

\bibitem[Pan83]{PansuCroissance}
Pierre Pansu.
\newblock Croissance des boules et des g{\'e}od{\'e}siques ferm{\'e}es dans les
  nilvari{\'e}t{\'e}s.
\newblock {\em Ergodic Theory and Dynamical Systems}, 3(03):415--445, 1983.

\bibitem[Pha12]{NPcusp}
T.~T\^{a}m~Nguy\^{e}n Phan.
\newblock Smooth (non)rigidity of cusp-decomposable manifolds.
\newblock {\em Commentarii Mathematici Helvetici}, 87(4):789--804, 2012.

\bibitem[Sch95]{SchwRank1}
Richard~Evan Schwartz.
\newblock The quasi-isometry classification of rank one lattices.
\newblock {\em Publications Math{\'e}matiques de l'Institut des Hautes
  {\'E}tudes Scientifiques}, 82(1):133--168, 1995.

\end{thebibliography}

\end{document}